\documentclass[a4paper,10pt]{article}
\usepackage[hmargin=2.0cm,vmargin=2.0cm]{geometry}
\usepackage[utf8x]{inputenc}

\usepackage{url}
\usepackage{graphicx}
\usepackage{moreverb}
\usepackage[vlined]{algorithm2e}
\usepackage{listings,color}
\usepackage{amsmath, amssymb,amsthm}
\usepackage{xspace}
\usepackage[pdftex, bookmarks=true, bookmarksnumbered=true]{hyperref}

\definecolor{_lightblue}{rgb}{0.3,0.3,1}
\definecolor{_gray}{gray}{.5}
\definecolor{_orange}{rgb}{0.97,0.5,0.04}
\definecolor{_gray}{rgb}{0.35,0.35,0.35}
\definecolor{_lightgray}{rgb}{0.93,0.93,0.93}
\definecolor{_blue}{rgb}{0.14, 0.43, 0.69}

\lstdefinelanguage{Sage}[]{Python}
{morekeywords={True,False,sage,with},
sensitive=true}

\hypersetup{
 pdfauthor = {Martin Albrecht and John Perry},
 pdftitle = {F45},
 colorlinks = true,
 linkcolor = _blue,
 citecolor = _blue,
 urlcolor = _blue,
}

\newcommand{\Sage}{Sage\xspace}

\newcommand\pring{\ensuremath{\mathcal R}}
\newcommand\sigset{\ensuremath{\mathcal S}}

\newcommand{\field}[1]{\ensuremath{\mathbb{#1}}\xspace}
\newcommand{\F}{\ensuremath{\field{F}}\xspace}

\newcommand{\N}{\field{N}}
\newcommand{\e}{\ensuremath{\mathbf{e}}\xspace}

\newcommand{\ideal}[1]{\langle {#1} \rangle}
\newcommand{\Mac}[1]{\ensuremath{\mathcal{M}^{acaulay}_{#1}}}
\newcommand{\LCM}{\ensuremath{\textsc{LCM}}\xspace}
\newcommand{\LM}{\ensuremath{\textsc{LM}}\xspace}
\newcommand{\LC}{\ensuremath{\textsc{LC}}\xspace}
\newcommand{\LT}{\ensuremath{\textsc{LT}}\xspace}

\newcommand{\sig}{\ensuremath{\textnormal{sig}}}
\newcommand{\idx}{\ensuremath{\textnormal{idx}}}
\newcommand{\poly}{\ensuremath{\textnormal{poly}}}

\newtheorem{theorem}{Theorem}[section]

\newtheorem{lemma}[theorem]{Lemma}
\newtheorem{proposition}[theorem]{Proposition}
\newtheorem{notation}[theorem]{Notation}

\newtheorem{definition}{Definition}[section]

\newenvironment{citeproof}{\vspace{0.3cm}\emph{Proof.}}{\vspace{0.3cm}}

\theoremstyle{definition}

\title{F4/5}
\author{Martin Albrecht \and John Perry}

\begin{document}

\maketitle

\begin{abstract}
We describe an algorithm to compute Gr\"obner bases which combines $F_4$-style reduction with the $F_5$ criteria. Both $F_4$ and $F_5$ originate in the work of Jean-Charles Faug\`ere \cite{F4,F5}, who has successfully computed many Gr\"obner bases that were previously considered intractable. Another description of a similar algorithm already exists in Gwenole Ars' dissertation~\cite{ars:thesis2005}; unfortunately, this is only available in French, and although an implementation exists, it is not made available for study. We not only describe the algorithm, we also direct the reader to a study implementation for the free and open source \Sage computer algebra system \cite{Sage}. We conclude with a short discussion of how the approach described here compares and contrasts with that of Ars' dissertation.
\end{abstract}

\section{Introduction}

This work describes and discusses Jean-Charles Faugère's $F_5$ algorithm. However, instead of presenting $F_5$ in the ``traditional'' fashion as is done in \cite{F5,Stegers2005,Gash2008}, a variant of $F_5$ in $F_4$-``style'' is presented. We refer to this variant as $F_{4/5}$.
The main differences between $F_{4/5}$ and $F_5$ are:
\begin{itemize}
 \item The two outermost loops are swapped (cf.~\cite{faugere:fse2007}), such that Algorithm~\ref{alg:f5} proceeds by degrees first and then by index of generators. $F_5$ proceeds by index of generators first and then by degrees.
 \item The polynomial reduction routines are replaced by linear algebra quite similar to matrix-$F_5$ (cf.~\cite{bardet-faugere-salvy:tech,faugere-ars-2004}).
 \item The lists \emph{Rules$_i$} are kept sorted at all times, which  matches matrix-$F_5$ closer and seems to improve performance slightly.
 \item Polynomial indices are reversed in Algorithm~\ref{alg:f5} compared to \cite{F5}. That is, we compute the Gröbner basis for the ideal $\ideal{f_0}$ first and not for the ideal $\ideal{f_{m-1}}$.
\end{itemize}

A study implementation of Algorithm~\ref{alg:f5} for Sage is available at
\begin{center}
\url{http://bitbucket.org/malb/algebraic_attacks/src/tip/f5_2.py}
\end{center}
and a study implementation of $F_5$ proper and variants is available at
\begin{center}
\url{http://bitbucket.org/malb/algebraic_attacks/src/tip/f5.py}.
\end{center}


\section{Background material}\label{sec:background}

Let $\pring = \F[x_0,\ldots,x_{n-1}]$ be a polynomial ring over the field $\F$. The goal of any $F_5$-class algorithm (including $F_{4/5}$)
is to compute a Gr\"obner basis of $f_0,\ldots,f_{m-1}\in\pring$ with respect to a given monomial ordering.

The distinguishing feature of $F_5$ is that it records part of a representation of each polynomial (or row) in terms of the input. This record is kept in a so-called \emph{signature}.

\begin{definition}[Signature]
Let $P^m$ be the free module over $\pring$ and let $\e_i$ be a canonical unit vector in $P^m$:  $\e_i = (0,\dots,0,1,0,\dots,0)$ where the $1$ is in the $i$-th position. A \textbf{signature} is any product $\sigma = t \cdot \e_i$, where $t$ is a monomial in $x_0,\ldots,x_{n-1}$. We denote by $\sigset$ the set of all signatures.
\end{definition}

We extend the monomial ordering on $\pring$ to $\sigset$.

\begin{definition}
Let $t\e_i$ and $u\e_j$ be signatures, we say that $t\e_i > u\e_j$ if
\begin{itemize}
 \item $i > j$ or
 \item $i = j$ and $t > u$.
\end{itemize}
\end{definition}

To each polynomial we associate a signature; this pair is called a \emph{labelled polynomial}. We are interested only in associating signatures with polynomials in a specific way.

\begin{definition}[Labelled Polynomial]
Let $\sigma\in\sigset$ and $f\in\pring$. We say that $(\sigma,f)$ is a
\textbf{labelled polynomial}. In addition, we say that $(\sigma,f)$ is \textbf{admissible} if there exist $h_0,\ldots,h_{m-1}\in\pring$
such that
\begin{itemize}
\item $f=h_0 f_0 + \cdots + f_{m-1} h_{m-1}$,
\item $h_{i+1} = \cdots = h_{m-1} = 0$, and
\item $\sigma = \LM(h_i) \e_i$.
\end{itemize}
\end{definition}

The following properties of admissible polynomials are trivial.

\begin{proposition}
Let $t,u,v$ be monomials and $f,g\in\pring$.
Assume that $(u\e_i,f)$ and $(v\e_j,g)$ are admissible.
Each of the following holds.
\begin{itemize}
\item[(A)] $(tu\e_i, tf)$ is admissible.
\item[(B)] If $i > j$, then $(u\e_i,f+g)$ is admissible.
\item[(C)] If $i = j$ and $u > v$, then $(u\e_i,f+g)$ is admissible.
\end{itemize}
\end{proposition}

In light of this fact, we can define the product of a monomial and a signature in a natural way. Let $t,u$ be monomials and $\sigma\in\sigset$ such that
$\sigma=u\e_i$ for some $i\in \N$. Then\[
t\cdot\sigma = tu\e_i.
\]

Whenever $F_5$ creates a labelled polynomial, it adds it to the global list $L$. Instead of passing around labelled polynomials, indices of $L$ are passed to subroutines. We thus identify a labelled polynomial $r$ with the natural number $i$ such that $L_i = r$. The algorithm's correctness and behaviour depends crucially on the assumption that all elements of $L$ are admissible. Thus all $F_5$-class algorithms ensure that this is the case at all times.

\begin{notation}
Let $r \in L$ and write $r= (t \cdot \e_i, p)$. We write
\begin{itemize}
 \item $\poly(r) = p$,
 \item $\sig(r) = t \cdot \e_i$, and
 \item $\idx(r) = i$.
\end{itemize}
\end{notation}

\begin{definition}
Let $a,b\in \N$ and suppose that $\sig(a)=u\e_i$ and $\sig(b)=v\e_j$.
Let $t_a = \LM(\poly(a))$, $t_b = \LM(\poly(b))$, and $$\sigma_{a,b} = \LCM(t_a,t_b)/t_a.$$
If $\sigma_{a,b} \sig(a) > \sigma_{b,a} \sig(b)$ then
the \textbf{naturally inferred signature} of the S-polynomial $S$ of $\poly(a)$ and $\poly(b)$ is $\sigma_{a,b}\cdot u\e_i$.
\end{definition}

From (B) and (C) above we can see that $(\sigma_{a,b}\cdot u\e_i,S)$ is admissible if $a$ and $b$ are admissible.

The following is proved in \cite{F5C}.

\begin{proposition}
Let $i,k\in \N$. Let $h_0,\ldots,h_{m-1}\in\pring$ such that $h_{i+1}=\ldots=h_{m-1}=0$ and $\sig(k)=\LM(h_i)\e_i$.
$\sig(k)$ is not the minimal signature of $\poly(k)$ if and only if there exists a syzygy $(z_0,\ldots,z_{m-1})\in P^m$ of $f_0,\ldots,f_{m-1}$ such that
\begin{itemize}
\item $\sig(k)$ is a signature of $z_0 f_0 + \cdots z_{m-1} f_{m-1}$;
\item if $t\e_j$ is the minimal signature of $\poly(k)$, then $h_k-z_k=0$ for all $k>j$ and $\LM(h_j-z_j)=t$. 
\end{itemize}
\end{proposition}



From this proposition it follows that we only need to consider S-polynomials with minimal signatures.

Suppose that all syzygies of $F$ are generated by trivial syzygies of the form $f_i \e_j - f_j \e_i$. If $\sig(k)$ is not minimal, then some multiple of a principal syzygy $m(f_i \e_j − f_j \e_i)$ has the same signature $\sig(k)$. This provides an easy test for such a non-minimal signature and thus reductions to zero. Since all syzygies are in the module of trivial syzygies, the signature must be a multiple of the leading monomial of a polynomial already in the basis.

\begin{theorem}[$F_5$ Criterion]
An S-polynomial with signature $t\e_i$ is redundant and can be discarded if there exists some $g$ with $\idx(g) < i$ such that $\LM(g) \mid t$.
\end{theorem}

Another application of the signatures consists in ``rewrite rules''.

\begin{definition}
A \textbf{rule} is any $(\sigma,k)\in\sigset\times\N$ such that $\sigma=\sig(k)$.
\end{definition}

The algorithm uses a global variable, $Rules$, which is a list of $m$ lists of \emph{rules}. We can view the elements of any $Rules_i$ in two ways.
\begin{itemize}
\item Each element of $Rules_i$ designates a ``canonical reductor'' for certain monomials, in the following sense.
Let $f, g_1, g_2\in\pring$ and assume that $\LM(g_1),\LM(g_2)\mid\LM(f)$ and $\idx(f)=\idx(g_1)=\idx(g_2)$. In a traditional algorithm to compute a Gr\"obner basis, the choice of whether to reduce $f$ by $g_1$ or by $g_2$ is ambiguous, and either may be done.
In $F_5$ class algorithms, by contrast, \emph{there is no such choice!} One \emph{must} reduce $\LM(f)$ by exactly one of the two, depending on
which appears later in $Rules_i$. A similar technique is used by involutive methods to compute Gr\"obner bases~\cite{involutive}.
For both methods, the restriction to one canonical reductor appears to improve performance dramatically.
\item Each element of $Rules_i$ corresponds to a ``simplification rule''; that is, a linear dependency already discovered. From the ``polynomial'' perspective, $(\sigma,k)\in Rules_i$ only if either $k < m$ or there exist $a,b\in\N$, $h_j\in\pring$, and monomials $t,u$ such that
\begin{itemize}
\item $S$ was first computed as the $S$-polynomial $t\cdot\poly(a) - u\cdot\poly(b)$  of $\poly(a)$ and $\poly(b)$;
\item $S = \sum_{j\neq k}h_j\cdot\poly(j) + \poly(k)$ with $\LM(h_j\poly(j))\leq\LM(S)$ for each $j$; and
\item $\sigma = \sig(k)$ is the naturally inferred signature of $S$.
\end{itemize}
In matrix-$F_5$, instead of starting from scratch from the original $f_i$ for each degree $d$, the matrix $\Mac{d-1}$ is used to construct the matrix $\Mac{d}$ in order to re-use the linear dependencies discovered at degree $d-1$. The same task is accomplished by the set of simplification rules in $Rules_i$, but instead of computing all multiples of the elements in $Rules_i$ we merely use it as a lookup table to replace a potential polynomial by an element from $L$ where reductions by smaller signatures were already performed.
\end{itemize}

Strictly speaking, any rule is somewhat redundant: if $(\sigma,k)\in Rules_i$ then we know that $\sigma=t\e_i$ for some monomial $t$. Hence it is sensible to store only $t$ rather than $\sigma$.

\section{Pseudocode}

We can now define the main loop of the $F_5$ algorithm (cf.~Algorithm~\ref{alg:f5}). This is similar to the main loop of $F_4$ except that:
\begin{itemize}
 \item for each input polynomial $f_i$ we create the labelled polynomial $(1\cdot\mathbf{e_i},\LC(f_i)^{-1}\cdot f_i)$, which is obviously admissible; and
 \item for each computed polynomial $f_i$, the rule $(\sig(i),i)$ is added to $Rules_{\idx(i)}$.
\end{itemize}

\begin{algorithm}
\caption{$F_{4/5}$} 
\KwIn{$F$ -- a list of homogeneous polynomials $f_0,\dots,f_{m-1}$}
\KwResult{a Gröbner basis for $F$}
\SetKwFunction{KwRed}{reduction}
\SetKw{KwAnd}{and}
\SetKw{KwWith}{with}
\Begin{
sort $F$ by total degree\;
$L, G, P \longleftarrow [], \varnothing, []$\;
\For{$0 \leq i < m$}{
  append $(1 \cdot \e_i, LC(f_i)^{-1} \cdot f_i)$ to $L$\;
  \textsc{Add Rule}$(1 \cdot \e_i, i)$\;
  $P \longleftarrow P \bigcup \{$\textsc{Update}$_{F5}(i,j,G): \forall j \in
G\}$\;
  add $i$ to $G$\;
}

\While{$P \neq \varnothing$}{
  $d \longleftarrow$ the minimal degree in $P$\;
  $P_d \longleftarrow$ all pairs with degree $d$\;
  $P \longleftarrow P\setminus P_d$\;
  $S \longleftarrow$ \textsc{S-Polynomials}$_{F5}$($P_d$)\;
  $\tilde{S} \longleftarrow$ \textsc{Reduction}$_{F5}$($S,G$)\;
  \For{$i \in \tilde{S}$}{
    $P \longleftarrow P \bigcup \{$\textsc{Update}$_{F5}(i,j,G): \forall j \in
G\}$\;
    add $i$ to $G$\;
  }
}
\Return{$\{\poly(f)\ |\ \forall f \in G\}$}\;
}
\label{alg:f5}
\end{algorithm}

The subroutine \textsc{Update}$_{F5}$ constructs a new critical pair for two labelled polynomials indexed in $L$. A critical pair in $F_5$ is represented the same way as a critical pair in $F_4$, except that the polynomials are replaced by indices to labelled polynomials.

Just like the routine \textsc{Update} in $F_4$ imposes the Buchberger criteria, \textsc{Update}$_{F5}$ imposes the $F_5$ criteria. These checks are:

\begin{itemize}

 \item Make sure that the multipliers that give rise to the components of the S-polynomial are not in the leading monomial ideal spanned by the leading monomials of the polynomials with index smaller than the S-polynomial component. This would imply that the natural signature which the algorithm would assign to the S-polynomial is not the minimal signature, and can be discarded by the $F_5$ criterion.

 \item Check whether a rule forbids generating one component of the S-polynomial. This has the same purpose as reusing $\Mac{d-1}$ for $\Mac{d}$ in matrix-$F_5$. If a component $u \cdot r$ of the S-polynomial is rewritable, this means that there is an element which can replace it which has probably had more reductions applied to it already. The element that rewrites the rewritable component was either already considered or will be considered in the future. Thus this avoids re-computation of the same linear combinations.

 \item Ensure that the signature of the resulting S-polynomial is the one that we would infer naturally. This should be the larger signature of the components; that is, that the labelled polynomial remains admissible.
\end{itemize}

\begin{algorithm}[htbp]
\caption{\textsc{Update}$_{F5}$}
\label{alg:critpair} 
\KwIn{$k$ -- an integer $0 \leq k < |L|$}
\KwIn{$l$ -- an integer $0 \leq l\neq k < |L|$}
\KwIn{$G$ -- a list of integers with elements $e$ such that $0 \leq e < |L|$}
\KwResult{the critical pair for $\poly(k)$ and $\poly(l)$, iff the $F_5$
criteria pass.}
\SetKw{KwAnd}{and}
\SetKw{KwOr}{or}
\Begin{
$t_k, t_l \longleftarrow \LT(\poly(k)), \LT(\poly(l))$\;
$t \longleftarrow \LCM(t_k,t_l)$\;
$u_k, u_l \longleftarrow t/t_k,t/t_l$\;
$(m_k, \e_k), (m_l, \e_l) \longleftarrow \sig(k), \sig(l)$\;


\If{\textsc{Top-reducible(}$u_k \cdot m_k$, \{$g_i \in G$: $\idx(g_i) <
\e_k$\}\textnormal{)}}{\Return\;}
\If{\textsc{Top-reducible(}$u_l \cdot m_l$, \{$g_i \in G$: $\idx(g_i) <
\e_l$\}\textnormal{)}}{\Return\;}

\If{\textsc{Rewritable}($u_k,k$) \KwOr
\textsc{Rewritable}($u_l,l$)}{\Return\;}

\If{$u_k \cdot \sig(k) < u_l \cdot \sig(l)$}{
  swap $u_k$ and $u_l$\;
  swap $k$ and $l$\;
}
\Return{$(t,u_k,k,u_l,l)$}\;
}
\end{algorithm}

The routine \textsc{S-Polynomials}$_{F5}$ first checks the rewritable criterion again, in case new elements have been created which would rewrite a component after creation of the critical pair. Then it computes the actual S-polynomials in such a way that only the part is computed which gives rise to the new signature. The subtraction of the other component and thus the cancellation of leading terms is delayed to the reduction routine. Indeed, \textsc{S-Polynomials}$_{F5}$ discards the component $(v,l)$ and relies on \textsc{Symbolic Preprocessing}$_{F5}$ to find a reductor for $u\cdot \poly(l)$.
We delay the rationale for this until the discussion of that algorithm; see below.

\begin{algorithm}[htbp]
\KwIn{$P$ -- a list of critical pairs}
\KwResult{a list of S-polynomials}
\SetKw{KwAnd}{and}
\SetKw{KwOr}{or}
\SetKw{KwContinue}{continue}
\Begin{
$S \longleftarrow \varnothing$\;
sort $P$ by increasing signature\;
\For{$(t,u,k,v,l) \in P$}{
  \If{\textsc{Rewritable(}$u,k$\textnormal{)} \KwOr
\textsc{Rewritable(}$v,l$\textsc{)}}{
   \KwContinue\;
  }
  add $(u, k)$ to $S$\;
}
sort $S$ by signatures\;
\Return{$S$}\; 
}
\caption{\textsc{S-Polynomials}$_{F5}$}
\label{alg:compspols} 
\end{algorithm}

The routine \textsc{Add Rule} simply adds an entry to the list $Rules_i$ encoding that the signature $\sigma$ corresponds to the labelled polynomial $k$.
Note, however, that $F_{4/5}$ sorts the list $Rules_i$ by $t$, while other versions of $F_5$ simply append new rules at the end of the list. The latter approach ensures that $Rules_i$ is sorted by degree of $t$, but it does not necessarily impose an ordering w.r.t. to the monomial ordering on $Rules_i$. 

\begin{algorithm}
\caption{\textsc{Add Rule}}
\KwIn{$\sigma$ --a signature}
\KwIn{$k$ -- an integer $0 \leq k < |L|$}
\Begin{
let $t$, $i$ be such that $t \cdot \e_i = \sigma$\;
insert $(t,k)$ into $Rules_i$ such that the order on $t$ is preserved\;
}
\label{alg:addrule} 
\end{algorithm}

The routine \textsc{Rewritable} determines whether $u\cdot\sig(k)$ is rewritable, as outlined in the Section~\ref{sec:background}.

\begin{algorithm}
\KwIn{$u$ -- a monomial}
\KwIn{$k$ -- an integer $0 \leq k < |L|$}
\KwResult{true iff $u \cdot \sig(k)$ is rewritable}
\SetKw{KwAnd}{and}
\Begin{
let $t$, $i$ be such that $t \cdot \e_i = \sig(k)$\;
\For{$|Rules_i| > ctr \geq 0$}{
  $(v,j) \longleftarrow Rules_i[ctr]$\;
  \If{$v\ |\ (u \cdot t)$}{
    \Return{$j\neq k$}\;
  }
}
\Return{false}\;
}
\caption{\textsc{Rewritable}}
\label{alg:rewritable} 
\end{algorithm}

\begin{algorithm}
\KwIn{$t$ -- a monomial}
\KwIn{$G$ -- a set of indices in $L$}
\KwResult{true iff $t$ is top-reducible by any element in $G$}
\Begin{
\For{$g \in G$}{
  \If{$\LM(\poly(g))\ |\ t$} {
   \Return{true}\;
  }
}
\Return{false}\;
}
\caption{\textsc{Top-reducible}}
\label{alg:topreducible} 
\end{algorithm}

Algorithm \textsc{Reduction}$_{F5}$ organises the reduction of the S-polynomials. It first calls \textsc{Symbolic\ Preprocessing}$_{F5}$ to determine which monomials and which polynomial multiples might be encountered while reducing the S-polynomials. The resulting list of polynomial multiples is sorted in decreasing order of their signatures, in order to avoid reducing a polynomial by another with a larger signature (a phenomenon called ``signature corruption'' which has catastrophic consequences on the computation of the basis). Reduction then calls \textsc{Gaussian\ Elimination}$_{F5}$, which transforms the list of polynomials into a matrix, performs Gaussian elimination without swapping rows or columns, then extracts the polynomials from the matrix.
``New'' polynomials in the system are identified by the fact that their leading monomials have changed from that of the polynomials in $F$: that is, a reduction of the leading monomial took place.
We add each new polynomial to the system, and create a new rule for this polynomial.

Sometimes, a reductor has signature larger than the polynomial that it would reduce. To avoid signature corruption, $F_5$ class algorithms consider this as another S-polynomial, and as a consequence generate a new polynomial. However, \textsc{Symbolic\ Preprocessing}$_{F5}$ cannot know beforehand whether this new polynomial is indeed necessary, so it does not generate a new rule, nor add it to $L$. This is done in \textsc{Reduction}$_{F5}$.

\begin{algorithm}
\KwIn{$S$ -- a list of S-polynomials indexed in $L$}
\KwIn{$G$ -- a list of polynomials indexed in $L$}
\KwResult{the top-reduced set $\tilde{S}$}
\SetKw{KwContinue}{continue}
\SetKw{KwAnd}{and}
\Begin{
$F,T \longleftarrow $\textsc{Symbolic\ Preprocessing}$_{F5}$($S,G$)\;
$\tilde{F} \longleftarrow$ \textsc{Gaussian Elimination}$_{F5}(F,T)$\;
$\tilde{F}^+ \longleftarrow \varnothing$\;
\For{$0 \leq k < |F|$}{
  $(u, i) \longleftarrow F_k$\;
  $\sigma \longleftarrow \sig(i)$\;
  \If{$u\cdot\LM(\poly(i)) = \LM(\tilde{F}_k)$}{\KwContinue\;}
  $\tilde{p} \longleftarrow \tilde{F}_k$\;
  append $(u \sigma,\tilde{p})$ to $L$; \tcp{Create new entry}
  \textsc{Add Rule}$(u \sigma,|L|-1)$\;
  \If{$\tilde{p} \neq 0$}{add $i$ to $\tilde{F}^+$\;}

}
\Return{$\tilde{F}^+$}\;
}
\caption{\textsc{Reduction}$_{F5}$}
\label{alg:reduction5} 
\end{algorithm}

\begin{algorithm}
\KwIn{$S$ -- a list of components of S-polynomials}
\KwIn{$G$ -- a list of polynomials indexed in $L$}
\KwResult{$F$ -- a list of labelled polynomials that \emph{might} be used during reduction of the S-polynomials of $S$}
\Begin{
$F \longleftarrow S$\;
$Done \longleftarrow \varnothing$\;
let $M'$ be the monomials of $\{\poly(k) \mid \forall k \in F\}$\;  
\While{$M' \neq Done$}{
  let $m$ be maximal in $M' \setminus Done$\;
  add $m$ to $Done$\;
  let $\sigma$ be minimal in $\left\{\sig(k)\ |\ k \in F \textrm{ and } m \textrm{ is a monomial of } \poly(k)\right\}$\;
  $t, k \longleftarrow$ \textsc{Find Reductor}$(m, \sigma, G, F)$\;
  \If{$t\neq 0$}{
   append $(t, k)$ to $F$\;
   add the monomials of $t\cdot\poly(k)$ to $M'$\;  
  }
}
sort $F$ by decreasing signature\;
\Return{$F, Done$}
}
\caption{\textsc{Symbolic Preprocessing}$_{F5}$}
\label{alg:symbolic_preprocessing5}
\end{algorithm}

\begin{algorithm}
\KwIn{$m$ -- a monomial}
\KwIn{$G$ -- a list of polynomials indexed in $L$}
\KwIn{$F$ -- a list of primary generators of $S$-poly\-nomials}
\SetKw{KwContinue}{continue}
\Begin{
\For{$k \in G$}{
 \If{$\LM(\poly(k)) \nmid m$}{\KwContinue\;}
 $u \longleftarrow m/\LM(\poly(k))$\;
 \If{$(u, k)\in F$}{\KwContinue\;}
 let $t\cdot\e_{i}$ be $\sig(k)$\;
 \If{\textsc{Top-reducible(}$u\cdot t$, \{$g \in G \mid \idx(g) < i$\}\textsc{)}}
 {\KwContinue}
 \If{\textsc{Rewritable(}$u,k$\textsc{)}}{\KwContinue\;}
 \Return{$u,k$}
}
\Return{0, -1}
}
\caption{\textsc{Find Reductor}}
\label{alg:find_reductor}
\end{algorithm}

The routine \textsc{Find Reductor} tries to find a reductor for a monomial $m$ with signature $\sigma$ in $G$. After checking the normal top reduction criterion it applies the same criteria to $t \cdot k$ as \textsc{Update}$_{F5}$ applies to the components of each S-polynomial.
However, we have added another check that does not appear in traditional pseudocode for F5: whether $u\cdot\LM(\poly(k))\in F$.

This returns us to a topic alluded to in the discussion of \textsc{S-Polynomials}$_{F5}$.
Recall that, in \textsc{S-Polynomials}$_{F5}$, we deferred the construction of $v\cdot\poly(l)$.
In most cases, there will be a choice of reductors for $u\cdot\poly(k)$; hypothetically, $v\cdot\poly(l)$ might not be the choice of \textsc{Symbolic Preprocessing}$_{F5}$. This would imply that the $S$-poly\-nomial of $\poly(k)$ and $\poly(l)$ might not be computed, even though it is necessary.
In fact, this cannot happen!
By way of contradiction, suppose that \textsc{Find Reductor} chooses $(t,j)$ to reduce $(u,k)$ and $(v,l)$ is not used to build the matrix:
then $$\LCM(\LM(\poly(j)),\LM(\poly(l))) \leq \LCM(\LM(\poly(k)),\LM(\poly(l))).$$
Since the algorithm proceeds by ascending degree, it must also be considering the critical pair for $\poly(j)$ and $\poly(l)$, if it did not do so at a lower degree. We consider two cases.
\begin{itemize}
\item Suppose that the algorithm rejected a generator of the $S$-poly\-nomial of $\poly(j)$ and $\poly(l)$; the criteria would clearly reject multiples of these generators as well.
This leads to a contradiction: either \textsc{Rewritable} would have rejected $v\cdot\poly(l)$, so that \textsc{S-Polynomials}$_{F5}$ would not have computed the $S$-poly\-nomial of $\poly(k)$ and $\poly(l)$, or \textsc{Find\ Reductor}$_{F5}$ would have rejected $t\cdot\poly(j)$ as a reductor.
\item Suppose instead that the $S$-poly\-nomial of $\poly(j)$ and $\poly(l)$ either has been computed, or is being computed at this degree.
These two possibilities also lead to a contradiction.
\begin{itemize}
\item If it is being computed \emph{at this degree}, then one of $(t,j)$ or $(v,l)$ already appears in $F$.
If $(t,j)$ appears, then the second if statement of \textsc{Find\ Reductor} precludes it from selecting $(t,k)$ as a reductor of $(u,k)$ instead of $(v,l)$.
\item If, on the other hand, the $S$-poly\-nomial of $\poly(j)$ and $\poly(l)$ was computed at a lower degree, then the new polynomial would have a signature that rewrites one of $t\cdot\poly(j)$ or $v\cdot\poly(l)$ --- so that the algorithm either cannot select $(t,j)$ as a reductor, or it deems $v\cdot\poly(l)$ rewritable, which means that it does not compute the $S$-poly\-nomial of $\poly(k)$ and $\poly(l)$!
\end{itemize}
\end{itemize}
The only way to avoid a contradiction is for the algorithm to include $v\cdot\poly(l)$ in the matrix: either because it is already in the matrix, or because it is selected as a reductor of $u\cdot\poly(k)$.
Therefore, the reformulated pseudocode does in fact compute all necessary $S$-poly\-nomials.

The algorithm \textsc{Gaussian\ Elimination}$_{F5}$ constructs a matrix $A$ whose entries $a_{ij}$ correspond to the coefficient of the $j$th monomial of the $i$th product listed in the input $F$. 
Subsequently, \textsc{Gaussian\ Elimination}$_{F5}$ computes a row-echelon reduction of the matrix, but in a straitjacketed sense: to respect the monomial ordering, we cannot swap columns, and to respect the signatures, we cannot swap rows, nor can we reduce lower rows (which have smaller signatures) by higher rows (which have larger signatures). As a result, each non-zero row has a unique pivot, but the appearance of the resulting matrix may not, in fact, be triangular. This is also why we must reset the index $i$ after any successful reduction to the top of the matrix, in case rows of higher signature can be reduced by the new row. 

Finally, \textsc{Gaussian\ Elimination}$_{F5}$ returns a list of polynomials corresponding to the rows of the matrix $A$. Strictly speaking, there is no need to expand those polynomials of $F$ whose leading monomials have \emph{not} changed, since \textsc{Reduction}$_{F5}$ will discard them anyway. Thus, a natural optimisation would be to return the matrix $A$ to \textsc{Reduction}$_{F5}$, determine in that procedure which rows of the matrix need to be expanded, and expand only them. We have chosen to expand all of $A$ in the pseudocode in order to encapsulate the matrix entirely within this procedure.

\begin{algorithm}
\caption{\textsc{Gaussian\ Elimination}$_{F5}$}
\label{alg:gaussian_elimination}
\KwIn{$F$ -- a list of pairs $(u,k)$ indicating that the product $u\cdot\poly(k)$ must be computed}
\KwIn{$T$ -- a list of all the monomials in $F$}
\KwResult{$\tilde{F}$ -- a list of labelled polynomials}
\SetKw{KwContinue}{continue}
\SetKw{KwBreak}{break}
\SetKw{KwAny}{any}
\Begin{
  $m, n \longleftarrow |F|, |T|$\;
  denote each $F_i$ by $(u_i,k_i)$\;
  let $A$ be the $m\times n$ matrix such that $a_{ij}$ is the coefficient of
  $T_j$ in $u_i\cdot\poly(k_i)$\;
  \For{$0 \leq c < n$}{
    \For{$0 \leq r < m$}{
      \If{$a_{rc} \neq 0$}{
        \tcp{Ensure that we are only reducing by leading terms}
        \lIf{\KwAny $a_{ri} \neq 0 \mid 0 \leq i < c$}{\KwContinue\;}
        rescale the row $r$ such that the entry $a_{rc}$ is 1\;
        \For(\tcp*[h]clear below){$r+1 \leq i < m$}{
          \If{$a_{ic} \neq 0$}{
             eliminate the entry $a_{ic}$ using the row $r$\;
          }
        }
        \KwBreak;
      }
    }
  }
  let $\tilde{F}=A\cdot T=\left[\sum_{j=0}^{n-1} a_{ij}\cdot t_i\right]_{i=0}^{m-1}$\;
  \Return{$\tilde{F}$}
}
\end{algorithm}

\section{Correctness}

Since $F_{4/5}$ follows the general structure of $F_4$ it is helpful to assert that $F_4$ is correct.

\begin{lemma}
\label{lem:f4_correct}
When $F_4$ terminates it returns a Gröbner basis.
\end{lemma}

\begin{citeproof}
See \cite{F4}. 
\end{citeproof}

However, in $F_{4/5}$ we apply the $F_5$ criteria instead of Buchberger's criteria. Thus, we need to prove that these criteria do not discard any S-polynomial which would be needed for a Gröbner basis computation.

\begin{lemma}[\cite{F5C}]
\label{lem:criteria_correct}
Assume that the main loop of Algorithm~\ref{alg:f5} terminates with output $G$. Let $\mathcal{G} = \{\poly(g) \mid g \in G\}$. If every $S$-polynomial $S$ of $\mathcal{G}$ satisfies (A) or (B) where
\begin{itemize}
 \item[(A)] $S$ reduces to zero with respect to $\mathcal{G}$\;
 \item[(B)] a component $u\cdot\poly(k)$ of $S$ satisfies
 \begin{itemize}
   \item[(B1)] $u\cdot\sig(k)$ is not the minimal signature of $u\cdot\poly(k)$; or
   \item[(B2)] $u\cdot\sig(k)$ is rewritable;
 \end{itemize}
\end{itemize}
then $\mathcal{G}$ is a Gröbner basis for $\ideal{f_0,\dots,f_{m-1}}$.
\end{lemma}

\begin{citeproof}
See \cite{F5C}.
There is one subtlety to be noted: here we order $Rules_i$ by signature.
An examination of the proof shows that this does not pose any difficulty for correctness.
\end{citeproof}

The other main differences between $F_4$ and $F_{4/5}$ is that we apply a variant of Gaussian elimination in $F_{4/5}$ to perform the reduction. However, as shown below this does not affect the set of leading monomials.

\begin{lemma}
\label{lem:gauss_correct}
Let $F$ be a set of polynomials in $P = \F[x_0,\dots,x_{n-1}]$. Let $\tilde{F}$ be the result of Gaussian elimination and $\tilde{F'}$ the result of Algorithm~\ref{alg:gaussian_elimination} (\textsc{Gaussian Elimination}$_{F5}$). We have that $\LM(\tilde{F})  = \LM(\tilde{F'})$.
\end{lemma}

\begin{proof}
Assume for contradiction that there is an element $f \in \tilde{F}$ with $\LM(f) \not\in \LM(\tilde{F'})$. This implies that there is a row $r$ in the coefficient matrix of $F$ corresponding to a polynomial $g$ which would reduce to $f$ in Gaussian elimination. Assume that this reduction is not allowed in Algorithm~\ref{alg:gaussian_elimination} because the necessary reductor is in a row
$r'$ below $r$. In that case Algorithm~\ref{alg:gaussian_elimination} will add the row $r$ to the row $r'$ (since $r$ has smaller signature than $r'$) and store the result in $r'$ producing the same addition and cancellation of leading terms. Thus only the row index of the result changes but the same additions are performed except for the clearance of the upper triangular matrix which does not affect leading terms.
\end{proof}

This allows us to prove that $F_{4/5}$ indeed computes a Gröbner basis if it terminates.

\begin{theorem}
\label{theorem:f45-correct}
If $F_{4/5}$ terminates and returns $g_0,\dots,g_{r-1}$ for the input $\{f_0,\dots,f_{m-1}\}$ then $g_0,\dots,g_{r-1}$ is a Gröbner basis for the ideal spanned by $f_0,\dots,f_{m-1}$ where $g_0,\dots,g_{r-1}$ and $f_0,\dots,f_{m-1}$ are homogeneous polynomials in $\F[x_0,\dots,x_{n-1}]$.
\end{theorem}

\begin{proof}
Lemma~\ref{lem:f4_correct} states that the general structure of the algorithm is correct; Lemma~\ref{lem:gauss_correct} states that the output of \textsc{Gaussian Elimination}$_{F5}$ is not worse than the output of Gaussian elimination in $F_4$ from a correctness perspective since all new leading monomials are included. Inspection of Algorithm~\ref{alg:reduction5} shows that it does return the set $\{f \in \tilde{F} \mid \LM(f) \not\in F\}$ as required for correctness of $F_4$-style algorithms. Lemma~\ref{lem:criteria_correct} states that the pairs discarded by \textsc{Update}$_{F5}$ are not needed to compute a Gröbner basis. The correctness of the discarding of reductors in Algorithm~\ref{alg:find_reductor} also follows from Lemma~\ref{lem:criteria_correct}. Thus, we conclude that $F_{4/5}$ computes a Gröbner basis if it terminates.
\end{proof}

However, Theorem~\ref{theorem:f45-correct} does not imply that $F_{4/5}$ terminates for all inputs. We note however, that there are no known counter examples. The difficulty with proving termination is due to the fact that the set $F$ might not contain all possible reductors since the routine \textsc{Find Reductor} might discard a reductor if it is rewritable. While Lemma~\ref{lem:criteria_correct} shows that this discarding does not affect the correctness, it does not show that the algorithm terminates because elements might be added to $G$ and $P$ which have leading terms already in $\LM(\{\poly(g) \mid g \in G\})$.

\section{Relationship to Ars' dissertation}

We briefly describe the differences between the algorithm outlined here and that in~\cite{ars:thesis2005}. We refer to the latter as $F_{5/Ars}$.

\begin{itemize}
\item $F_{5/Ars}$ takes as input not only $F$, but also a function $\mathcal{S}$el to select critical pairs (cf.~\cite{F4}), whereas $F_{4/5}$ always selects pairs according to lowest degree of the \LCM. In this case, $F_{5/Ars}$ is more general, but note that the description of $F_4$ in~\cite{F4} claims that the most efficient method to select critical pairs is, in general, by lowest degree of the LCM.
\item $F_{5/Ars}$ uses two functions to update two lists of critical pairs:
  \begin{itemize}
  \item \textsc{Update1} is used to estimate the degree of termination (more correctly translated the \emph{degree of regularity --- degr{\'e} de regularit{\'e}}) and relies on Buchberger's LCM criterion. The critical pairs computed here are stored in a set $P$, but are never used to compute any polynomials, only to estimate the degree of termination.
  \item \textsc{Update2} is used to compute critical pairs that \emph{are} used to generate polynomials, and is comparable to \textsc{Update}$_{F5}$ here. In addition to the indices of two labelled polynomials and the set of indices of computed polynomials, \textsc{Update2} requires the list of previously computed critical pairs, and the estimated degree of termination. It discards critical pairs whose signatures are top-reducible by polynomials of lower index (the F5 criterion), as well as those whose degrees are larger than the estimated degree of termination.

\item Naturally, one wonders whether the estimated degree of termination is correct. The degree is estimated in the following way: any critical pair that passes Buchberger's second criterion is added to $P$, and the degree of termination is estimated as the largest degree of a critical pair in $P$.
    
    The reason such a method might be necessary in general is that no proof of termination exists for the F5 algorithms, not even in special cases~\cite{Gash2008}. The difficulty lies in the fact that F5 short-circuits many top-reductions in order to respect the criteria and the signatures (see \textsc{Symbolic\ Preprocessing} and \textsc{Find\ Reductor}). For various reasons, the redundant polynomials that result from this cannot be merely discarded --- some of their critical pairs are \emph{not} redundant --- but applying Buchberger's second criterion should allow one to determine the point at which all critical pairs are redundant.
    
    Note that a similar method to determine a degree of termination is given in~\cite{ederperry:f5+}, and is proven in detail. Each method has advantages over the other (one is slightly faster; the other computes a lower degree), and $F_{4/5}$ can be modified easily to work with either.
  \end{itemize}
\item $F_{5/Ars}$ adds both components of S-polynomials to the list of polynomials scheduled for reduction by Gaussian elimination. $F_{4/5}$ only adds the component with the bigger signature and relies on \textsc{Symbolic Preprocessing}$_{F5}$ to find a reductor for the leading term. This potentially allows for a reductor which had more reductions applied to it already.
\end{itemize}

Thus, the algorithms are essentially equivalent.

\section{A Small Example Run of \texorpdfstring{$F_{4/5}$}{F45}}

We consider the ideal $\ideal{x^{2} y - z^{2} t, x z^{2} - y^{2} t, y z^{3} - x^{2} t^{2}} \in \F_{32003}[x, y, z, t]$ with the degree reverse lexicographical monomial ordering.

After the initialisation $G$ contains three elements $$(\e_0, xz^2 - y^2t), (\e_1, x^2y - z^2t),(\e_2, yz^3 - x^2t^2)$$ and $P$ contains the three pairs $$(x^2yz^2, z^2, 1, xy, 0), (xyz^3, x, 2, yz, 0), (x^2yz^3, x^2, 2, z^3, 1).$$

At degree $d=5$ the algorithm selects the pairs $(x^2yz^2, z^2, 1, xy, 0)$ and $(xyz^3, x, 2, yz, 0)$ of which both survive the $F_5$ criteria. These generate two new labelled polynomials $L_3 = (x\e_2, xyz^3 - x^3t^2)$ and $L_4 =(z^2\e_1, x^2yz^2 - z^4t 4)$. These reduce to $y^3zt - x^3t^2$ and $xy^3t - z^4t$ respectively and are returned by \textsc{Reduction}$_{F5}$.

At degree $d=6$ the algorithm selects the pairs $(x^2yz^3, x^2, 2, z^3, 1)$ and $(xy^3zt, x, 3, z, 4)$ of which only the pair $(x^2yz^3, x^2, 2, z^3, 1)$ survives the $F_5$ criteria. This pair generates a new labelled polynomial $L_5 = (x^2\e_2,xy^3zt - x^4t^2)$ which reduces to $z^5t - x^4t^2$ and is returned by  \textsc{Reduction}$_{F5}$.

At degree $d=7$ the algorithm selects the critical pairs $$(xy^3z^2t, z^2, 4, y^3t, 0), (xy^3z^2t, xz, 3, y^3t, 0), (x^2y^3zt, x^2, 3, y^2zt, 1), (xz^5t, x, 5, z^3t, 0)$$ of which $(xy^3z^2t, z^2, 4, y^3t, 0)$ and $(xz^5t, x, 5, z^3t, 0)$ survive the $F_5$ criteria. These pairs generate two new labelled polynomials $L_6 = (x^3\e_2, xz^5t - x^5t^2)$ and $L_7 = (z^4\e_1, xy^3z^2t - z^6t)$. \textsc{Reduction}$_{F5}$ these reduce to $x^5t^2 - z^2t^5$ and $z^6t - y^5t^2$. However, \textsc{Reduction}$_{F5}$ also returns a third polynomial in order to preserve signatures, that is $L_8 = (x^2z\e_2, y^5t^2 - x^4zt^2)$.

At degree $d=8$ the algorithm selects the pair $(yz^6t, z^3t, 2, y, 7)$ which survives the $F_5$ criteria. This pair generates a new labelled polynomial $L_9 = 
(z^3t\e_2, yz^6t - x^2z^3t^3)$ which reduces to $y^6t^2 - xy^2zt^4$.

Then the algorithm terminates.


\bibliographystyle{alpha}
\bibliography{literature}

\end{document}